\documentclass[a4paper,11pt]{article}
\usepackage{amsmath}
\usepackage{amssymb}
\usepackage{theorem}
\usepackage{pstricks}
\usepackage{euscript}
\usepackage{epic,eepic}
\usepackage{graphicx}
\PassOptionsToPackage{normalem}{ulem}
\newcommand{\email}[1]{\href{mailto:#1}{\nolinkurl{#1}}}
\renewcommand{\leq}{\ensuremath{\leqslant}}
\renewcommand{\geq}{\ensuremath{\geqslant}}
\newcommand{\minimize}[2]{\ensuremath{\underset{\substack{{#1}}}%
{\mathrm{minimize}}\;\;#2 }}

\newcommand{\Frac}[2]{\displaystyle{\frac{#1}{#2}}} 
\newcommand{\menge}[2]{\big\{{#1} \mid {#2}\big\}}

\newcommand{\emp}{\ensuremath{{\varnothing}}}

\newcommand{\scal}[2]{\langle{#1}\mid {#2} \rangle}

\newcommand{\exi}{\ensuremath{\exists\,}}
   
\newcommand{\zerounf}{\ensuremath{\left]0,1\right]}}   
\newcommand{\HH}{\ensuremath{\mathcal H}}
\newcommand{\GG}{\ensuremath{\mathcal G}}

\newcommand{\SL}{\ensuremath{\EuScript S}\,}
\newcommand{\BP}{\ensuremath{\EuScript P}}

\newcommand{\RR}{\ensuremath{\mathbb R}}

\newcommand{\RP}{\ensuremath{\left[0,+\infty\right[}}
\newcommand{\RPP}{\ensuremath{\,\left]0,+\infty\right[}}
\newcommand{\RX}{\ensuremath{\,\left]-\infty,+\infty\right]}}

\newcommand{\RXX}{\ensuremath{\left[-\infty,+\infty\right]}}
\newcommand{\NN}{\ensuremath{\mathbb N}}

\newcommand{\prox}{\ensuremath{\operatorname{prox}}}

\newcommand{\inte}{\ensuremath{\operatorname{int}}}

\newcommand{\RPX}{\ensuremath{{[0,\pinf]}}}

\newcommand{\argmin}{\ensuremath{\operatorname{argmin}}}

\newcommand{\gra}{\ensuremath{\operatorname{gra}}}
\newcommand{\conv}{\ensuremath{\operatorname{conv}}}
\newcommand{\cconv}{\ensuremath{\overline{\operatorname{conv}}\,}}

\newcommand{\Fix}{\ensuremath{\operatorname{Fix}}}
\newcommand{\Id}{\ensuremath{\operatorname{Id}}}

\newcommand{\weakly}{\ensuremath{\rightharpoonup}}

\newcommand{\pinf}{\ensuremath{+\infty}}
\newtheorem{theorem}{Theorem}[section]
\newtheorem{lemma}[theorem]{Lemma}

\newtheorem{corollary}[theorem]{Corollary}
\newtheorem{proposition}[theorem]{Proposition}
\newtheorem{definition}[theorem]{Definition}
\theoremstyle{plain}{\theorembodyfont{\rmfamily}
}
\theoremstyle{plain}{\theorembodyfont{\rmfamily}
}
\theoremstyle{plain}{\theorembodyfont{\rmfamily}
}
\theoremstyle{plain}{\theorembodyfont{\rmfamily}
\newtheorem{example}[theorem]{Example}}
\theoremstyle{plain}{\theorembodyfont{\rmfamily}
\newtheorem{problem}[theorem]{Problem}}
\theoremstyle{plain}{\theorembodyfont{\rmfamily}
\newtheorem{remark}[theorem]{Remark}}
\theoremstyle{plain}{\theorembodyfont{\rmfamily}
}

\numberwithin{equation}{section}

\begin{document}
\title{\sffamily Variable Metric Quasi-Fej\'er 
Monotonicity\thanks{Contact author: 
P. L. Combettes, \email{plc@math.jussieu.fr},
phone: +33 1 4427 6319, fax: +33 1 4427 7200.
The work of B$\grave{\text{\u{a}}}$ng C\^ong V\~u was 
partially supported by Grant 102.01-2012.15 of the
Vietnam National Foundation for Science and Technology Development
(NAFOSTED).}}
\author{Patrick L. Combettes and 
B$\grave{\text{\u{a}}}$ng C\^ong V\~u\\[4mm]
\small UPMC Universit\'e Paris 06\\
\small Laboratoire Jacques-Louis Lions -- UMR CNRS 7598\\
\small 75005 Paris, France\\
\small\email{plc@math.jussieu.fr}, \email{vu@ljll.math.upmc.fr} 
}
\bigskip

\date{~}
\maketitle

\begin{abstract}
The notion of quasi-Fej\'er monotonicity has proven to be an 
efficient tool to simplify and unify the convergence analysis 
of various algorithms arising in applied nonlinear analysis. 
In this paper, we extend this notion in the context of 
variable metric algorithms, whereby the underlying norm 
is allowed to vary at each iteration. Applications to 
convex optimization and inverse problems are demonstrated.
\end{abstract}

{\bfseries Keywords}: 
convex feasibility problem,
convex optimization,
Hilbert space,
inverse problems,
proximal Landweber method,
proximal point algorithm,
quasi-Fej\'er sequence,
variable metric.

{\bf Mathematics Subject Classifications (2010)} 
65J05, 90C25, 47H09. 

\section{Introduction} 

Let $C$ be a nonempty closed subset of the Euclidean space $\RR^N$ 
and let $y$ be a point in its complement. In 1922, Fej\'er 
\cite{Feje22} considered the problem of finding a point
$x\in\RR^N$ such that $(\forall z\in C)$ $\|x-z\|<\|y-z\|$.
Based on this work, the term Fej\'er-monotonicity was coined
in \cite{Motz54} in connection with sequences 
$(x_n)_{n\in\NN}$ in $\RR^N$ that satisfy
\begin{equation}
\label{e:fejer}
(\forall z\in C)(\forall n\in\NN)\quad\|x_{n+1}-z\|\leq\|x_n-z\|.
\end{equation}
This concept was later broadened to that of quasi-Fej\'er
monotonicity in \cite{Ermo68} by relaxing \eqref{e:fejer} to 
\begin{equation}
\label{e:qfejer}
(\forall z\in C)(\forall n\in\NN)\quad\|x_{n+1}-z\|^2
\leq\|x_n-z\|^2+\varepsilon_n,
\end{equation}
where $(\varepsilon_n)_{n\in\NN}$ is a summable sequence in $\RP$.
These notions have proven to be remarkably useful in 
simplifying and unifying the convergence analysis of a large 
collection of algorithms arising in hilbertian nonlinear analysis, 
see for instance \cite{Baus96,Livre1,Else01,Opti04,Ency09,Ere69b,%
Erem09,Poly87,Raik69,Scho91} and the references therein.
In recent years, there have been attempts to generalize standard
algorithms such as those discussed in the above references by
allowing the underlying metric to vary over the course of the
iterations, e.g., \cite{Bonn95,Burk99,Chen97,Varm12,Loti09,Pare08}.
In order to better understand the
convergence properties of such algorithms and lay the ground for 
further developments, we extend in the present paper the 
notion of quasi-Fej\'er monotonicity to the context of variable 
metric iterations in general Hilbert spaces and investigate its 
properties. 

Our notation and preliminary results are presented in
Section~\ref{sec:2}. The notion of variable metric quasi-Fej\'er
monotonicity is introduced in Section~\ref{sec:3}, where weak and
strong convergence results are also established. In
Section~\ref{sec:4}, we focus on the special case
when, as in \eqref{e:qfejer}, monotonicity is with respect to 
the squared norms.
Finally, we illustrate the potential of these tools in the analysis
of variable metric convex feasibility algorithms in
Section~\ref{sec:5} and in the design of algorithms for solving
inverse problems in Section~\ref{sec:6}.

\section{Notation and technical facts} 
\label{sec:2}

Throughout, $\HH$ is a real Hilbert space, $\scal{\cdot}{\cdot}$ is
its scalar product and $\|\cdot\|$ the associated norm. The symbols 
$\weakly$ and $\to$ denote respectively weak and strong convergence,
$\Id$ denotes the identity operator, and $B(z;\rho)$ denotes the
closed ball of center $z\in\HH$ and radius $\rho\in\RPP$ ; 
$\SL(\HH)$ is the space of self-adjoint bounded linear operators 
from $\HH$ to $\HH$. The Loewner partial ordering 
on $\SL(\HH)$ is defined by
\begin{equation}
\label{e:loewner}
(\forall L_1\in\SL(\HH))(\forall L_2\in\SL(\HH))\quad
L_1\succcurlyeq L_2\quad\Leftrightarrow\quad(\forall x\in\HH)\quad
\scal{L_1x}{x}\geq\scal{L_2x}{x}.
\end{equation}
Now let $\alpha\in\RP$, set
\begin{equation}
\BP_{\alpha}(\HH)=\menge{L\in\SL(\HH)}{L\succcurlyeq\alpha\Id},
\end{equation}
and fix $W\in\BP_{\alpha}(\HH)$. We define a semi-scalar product and
a semi-norm (a scalar product and a norm if $\alpha>0$) by
\begin{equation}
\label{e:guad2012p}
(\forall x\in\HH)(\forall y\in\HH)\quad\scal{x}{y}_W=\scal{Wx}{y}
\quad\text{and}\quad\|x\|_W=\sqrt{\scal{Wx}{x}}.
\end{equation}
Let $C$ be a nonempty subset of $\HH$, let 
$\alpha\in\RPP$, and let $W\in\BP_\alpha(\HH)$. The interior of $C$
is $\inte C$, the distance function of $C$ is $d_C$, and the 
convex envelope of $C$ is $\conv C$, with closure $\cconv C$. If 
$C$ is closed and convex, the projection operator onto $C$ 
relative to the metric induced by $W$ in \eqref{e:guad2012p} is
\begin{equation}
P_C^W\colon\HH\to C\colon 
x\mapsto\underset{y\in C}{\argmin}\:\|x-y\|_W.
\end{equation}
We write $P_C^{\Id}=P_C$.  Finally, $\ell_+^1(\NN)$ denotes 
the set of summable sequences in $\RP$.

\begin{lemma}
\label{l:kjMMXII}
Let $\alpha\in\RPP$, let $\mu\in\RPP$, and let $A$ and $B$ be 
operators in $\SL(\HH)$ such that 
$\mu\Id\succcurlyeq A\succcurlyeq B\succcurlyeq\alpha\Id$. Then
the following hold.
\begin{enumerate}
\item
\label{l:kjMMXII-i}
$\alpha^{-1}\Id\succcurlyeq B^{-1}\succcurlyeq A^{-1}\succcurlyeq
\mu^{-1}\Id$. 
\item
\label{l:kjMMXII-ii}
$(\forall x\in\HH)$ $\scal{A^{-1}x}{x}\geq\|A\|^{-1}\|x\|^2$.
\item
\label{l:kjMMXII-iii}
$\|A^{-1}\|\leq\alpha^{-1}$.
\end{enumerate}
\end{lemma}
\begin{proof}
These facts are known \cite[Section~VI.2.6]{Kato80}.
We provide a simple convex-analytic proof.

\ref{l:kjMMXII-i}:
It suffices to show that $B^{-1}\succcurlyeq A^{-1}$. 
Set $(\forall x\in\HH)$ $f(x)=\scal{Ax}{x}/2$ and
$g(x)=\scal{Bx}{x}/2$. The conjugate of $f$ is
$f^*\colon\HH\to\RXX\colon u\mapsto\sup_{x\in\HH}
\big(\scal{x}{u}-f(x)\big)=\scal{A^{-1}u}{u}/2$
\cite[Proposition~17.28]{Livre1}. Likewise,
$g^*\colon\HH\to\RXX\colon u\mapsto\scal{B^{-1}u}{u}/2$.
Since, $f\geq g$, we have $g^*\geq f^*$, hence the result.

\ref{l:kjMMXII-ii}:
Since $\|A\|\Id\succcurlyeq A$, \ref{l:kjMMXII-i} yields
$A^{-1}\succcurlyeq\|A\|^{-1}\Id$.

\ref{l:kjMMXII-iii}: We have $A^{-1}\in\SL(\HH)$ and, by 
\ref{l:kjMMXII-i}, $(\forall x\in\HH)$ 
$\|x\|^2/\alpha\geq\scal{A^{-1}x}{x}$. Hence, upon taking the 
supremum over $B(0;1)$, we obtain $1/\alpha\geq\|A^{-1}\|$. 
\end{proof}

\begin{lemma} {\rm\cite[Lemma~2.2.2]{Poly87}}
\label{l:7}
Let $(\alpha_n)_{n\in\NN}$ be a sequence in $\RP$, 
let $(\eta_n)_{n\in\NN}\in\ell_+^1(\NN)$, and
let $(\varepsilon_n)_{n\in\NN}\in\ell_+^1(\NN)$ be such that 
$(\forall n\in\NN)$ 
$\alpha_{n+1}\leq(1+\eta_n)\alpha_n+\varepsilon_n$.
Then $(\alpha_n)_{n\in\NN}$ converges.
\end{lemma}

The following lemma extends the classical property that 
a uniformly bounded monotone sequence of operators in 
$\SL(\HH)$ converges pointwise \cite[Th\'eor\`eme~104.1]{Ries68}.

\begin{lemma}
\label{l:ES175}
Let $\alpha\in\RPP$, let $(\eta_n)_{n\in\NN}\in\ell_+^1(\NN)$, 
and let $(W_n)_{n\in\NN}$ be a sequence in $\BP_{\alpha}(\HH)$ 
such that $\mu=\sup_{n\in\NN}\|W_n\|<\pinf$. Suppose that one 
of the following holds.
\begin{enumerate}
\item
\label{l:ES175i}
$(\forall n\in\NN)$ $(1+\eta_n)W_n\succcurlyeq W_{n+1}$.
\item
\label{l:ES175ii}
$(\forall n\in\NN)$ $(1+\eta_n)W_{n+1}\succcurlyeq W_n$.
\end{enumerate}
Then there exists $W\in\BP_{\alpha}(\HH)$ such that 
$W_n\to W$ pointwise.
\end{lemma}
\begin{proof}
\ref{l:ES175i}: 
Set $\tau=\prod_{n\in\NN}(1+\eta_n)$, $\tau_0=1$, and, for every 
$n\in\NN\smallsetminus\{0\}$, $\tau_n=\prod_{k=0}^{n-1}(1+\eta_k)$.
Then $\tau_n\to\tau<\pinf$ \cite[Theorem~3.7.3]{Knop56} and
\begin{equation}
\label{e:refaire1}
(\forall n\in\NN)\quad
\mu\Id\succcurlyeq W_n\succcurlyeq\alpha\Id\quad\text{and}\quad
\tau_{n+1}=\tau_n(1+\eta_n).
\end{equation}
Now define
\begin{equation}
\label{e:fema}
(\forall n\in\NN)(\forall m\in\NN)\quad
W_{n,m}=\frac{1}{\tau_n}W_n-\frac{1}{\tau_{n+m}}W_{n+m}.
\end{equation}
Then we derive from \eqref{e:refaire1} that 
\begin{align}
\label{e:KJrapture}
(\forall n\in\NN)(\forall m\in\NN\smallsetminus\{0\})
(\forall x\in\HH)\quad 0
&=\frac{1}{\tau_n}\scal{W_nx}{x}-\frac{1}{\tau_{n+m}}
\prod_{k=n}^{n+m-1}(1+\eta_k)\scal{W_nx}{x}\nonumber\\
&\leq\frac{1}{\tau_n}\scal{W_nx}{x}-\frac{1}{\tau_{n+m}}
\scal{W_{n+m}x}{x}\nonumber\\
&=\scal{W_{n,m}x}{x}\nonumber\\
&\leq\frac{1}{\tau_n}\scal{W_nx}{x}\nonumber\\
&\leq\scal{W_nx}{x}\nonumber\\
&\leq\mu\|x\|^2.
\end{align}
Therefore
\begin{equation}
\label{e:poleshift}
(\forall n\in\NN)(\forall m\in\NN)\quad
W_{n,m}\in\BP_{0}(\HH)\quad\text{and}\quad
\|W_{n,m}\|\leq\mu.
\end{equation}
Let us fix $x\in\HH$. By assumption, $(\forall n\in\NN)$ 
$\|x\|_{W_{n+1}}^2\leq(1+\eta_n)\|x\|_{W_n}^2$.
Hence, by Lemma~\ref{l:7}, $(\|x\|_{W_n}^2)_{n\in\NN}$
converges. In turn, $(\tau_n^{-1}\|x\|_{W_n}^2)_{n\in\NN}$
converges, which implies that
\begin{equation}
\label{e:refaire2}
\|x\|^2_{W_{n,m}}=\scal{W_{n,m}x}{x}=\frac{1}{\tau_n}
\|x\|_{W_n}^2-\frac{1}{\tau_{n+m}}\|x\|_{W_{n+m}}^2
\to 0\quad\text{as}\quad n,m\to\pinf.
\end{equation}
Therefore, using \eqref{e:poleshift}, Cauchy-Schwarz for the 
semi-norms $(\|\cdot\|_{W_{n,m}})_{(n,m)\in\NN^2}$, and 
\eqref{e:refaire2}, we obtain
\begin{align}
\|W_{n,m}x\|^4
&=\scal{x}{W_{n,m}x}_{W_{n,m}}^2\nonumber\\
&\leq\|x\|^2_{W_{n,m}}\,\|W_{n,m}x\|^2_{W_{n,m}}\nonumber\\
&\leq\|x\|^2_{W_{n,m}}\,\mu^3\|x\|^2\nonumber\\
&\to 0\quad\text{as}\quad n,m\to\pinf.
\end{align}
Thus, we derive from \eqref{e:fema} that 
$(\tau_n^{-1}W_n x)_{n\in\NN}$ is a Cauchy sequence.
Hence, it converges strongly, and so does $(W_nx)_{n\in\NN}$. 
If we call $Wx$ the limit of $(W_nx)_{n\in\NN}$, the above
construction yields the desired operator $W\in\BP_{\alpha}(\HH)$.

\ref{l:ES175ii}: 
Set $(\forall n\in\NN)$ $L_n=W_n^{-1}$. It follows from 
Lemma~\ref{l:kjMMXII}\ref{l:kjMMXII-i}\&\ref{l:kjMMXII-iii} that 
$(L_n)_{n\in\NN}$ lies in 
$\BP_{1/\mu}(\HH)$, $\sup_{n\in\NN}\|L_n\|\leq 1/\alpha$, 
and $(\forall n\in\NN)$ $(1+\eta_n)L_n\succcurlyeq L_{n+1}$. 
Hence, appealing to \ref{l:ES175i}, there exists 
$L\in\BP_{1/\mu}(\HH)$ such that $\|L\|\leq 1/\alpha$ and 
$L_n\to L$ pointwise. Now let $x\in\HH$, and set 
$W=L^{-1}$ and $(\forall n\in\NN)$ $x_n=L_n(Wx)$. 
Then $W\in\BP_{\alpha}(\HH)$ and $x_n\to L(Wx)=x$. Moreover,
$\|W_nx-Wx\|=\|W_n(x-x_n)\|\leq\mu\|x_n-x\|\to 0$.
\end{proof}

\section{Variable metric quasi-Fej\'er monotone sequences}
\label{sec:3}

Our paper hinges on the following extension of \eqref{e:qfejer}.

\begin{definition}
\label{d:vmqf}
Let $\alpha\in\RPP$, let $\phi\colon\RP\to\RP$, let 
$(W_n)_{n\in\NN}$ be a sequence in $\BP_{\alpha}(\HH)$, let $C$ 
be a nonempty subset of $\HH$, and let $(x_n)_{n\in\NN}$ be a 
sequence in $\HH$. Then $(x_n)_{n\in\NN}$ is:
\begin{enumerate}
\item
$\phi$-quasi-Fej\'er monotone with respect to the target set $C$ 
relative to $(W_n)_{n\in\NN}$ if
\begin{multline}
\label{e:vmqf1}
\big(\exi(\eta_n)_{n\in\NN}\in\ell_+^1(\NN)\big)
\big(\forall z\in C\big)\big(\exi(\varepsilon_n)_{n\in\NN}\in
\ell_+^1(\NN)\big)(\forall n\in\NN)\\
\phi(\|x_{n+1}-z\|_{W_{n+1}}) 
\leq (1+\eta_n)\phi(\|x_n-z\|_{W_n})+\varepsilon_n;
\end{multline}
\item
stationarily $\phi$-quasi-Fej\'er monotone with respect to the
target set $C$ relative to $(W_n)_{n\in\NN}$ if
\begin{multline}
\label{e:vmqf2}
\big(\exi(\varepsilon_n)_{n\in\NN}\in\ell_+^1(\NN)\big)
\big(\exi(\eta_n)_{n\in\NN}\in\ell_+^1(\NN)\big)(\forall z\in C)
(\forall n\in\NN)\\
\phi(\|x_{n+1}-z\|_{W_{n+1}}) 
\leq (1+\eta_n)\phi(\|x_n-z\|_{W_n})+\varepsilon_n.
\end{multline}
\end{enumerate}
\end{definition}

We start with basic properties.

\begin{proposition}
\label{p:1} 
Let $\alpha\in\RPP$, let $\phi\colon\RP\to\RP$ be strictly 
increasing and such that $\lim_{t\to\pinf}\phi(t)=\pinf$, 
let $(W_n)_{n\in\NN}$ be in $\BP_{\alpha}(\HH)$, let $C$ be 
a nonempty subset of $\HH$, and let $(x_n)_{n\in\NN}$ be a 
sequence in $\HH$ such that \eqref{e:vmqf1} is satisfied. 
Then the following hold.
\begin{enumerate}
\item 
\label{p:1i}
Let $z\in C$. Then $(\|x_n-z\|_{W_n})_{n\in\NN}$ converges.
\item
\label{p:1ii} 
$(x_n)_{n\in\NN}$ is bounded.
\end{enumerate}
\end{proposition}
\begin{proof}
\ref{p:1i}: Set $(\forall n\in\NN)$ $\xi_n=\|x_n-z\|_{W_n}$.
It follows from \eqref{e:vmqf1} and Lemma~\ref{l:7} that 
$(\phi(\xi_n))_{n\in\NN}$ converges, say 
$\phi(\xi_n)\to\lambda$. In turn, since 
$\lim_{t\to\pinf}\phi(t)=\pinf$, $(\xi_n)_{n\in\NN}$ is bounded 
and, to show that it converges, it suffices to show that it cannot
have two distinct cluster points. Suppose to the contrary that
we can extract two subsequences $(\xi_{k_n})_{n\in\NN}$ and 
$(\xi_{l_n})_{n\in\NN}$ such that $\xi_{k_n}\to\eta$ and 
$\xi_{l_n}\to\zeta>\eta$, and fix 
$\varepsilon\in\left]0,(\zeta-\eta)/2\right[$. Then, for $n$
sufficiently large, $\xi_{k_n}\leq\eta+\varepsilon<
\zeta-\varepsilon\leq\xi_{l_n}$ and, since $\phi$ is strictly
increasing, $\phi(\xi_{k_n})\leq\phi(\eta+\varepsilon)<
\phi(\zeta-\varepsilon)\leq\phi(\xi_{l_n})$. Taking the limit as
$n\to\pinf$ yields $\lambda\leq\phi(\eta+\varepsilon)<
\phi(\zeta-\varepsilon)\leq\lambda$, which is impossible.

\ref{p:1ii}: Let $z\in C$. Since $(W_n)_{n\in\NN}$ lies in 
$\BP_{\alpha}(\HH)$, we have
\begin{equation}
\label{e:boo}
(\forall n\in\NN)\quad\alpha\|x_n-z\|^2\leq\scal{x_n-z}{W_n(x_n-z)}
=\|x_n-z\|_{W_n}^{2}.
\end{equation}
Hence, since \ref{p:1i} asserts that $(\|x_n-z\|_{W_n})_{n\in\NN}$
is bounded, so is $(x_n)_{n\in\NN}$.
\end{proof}

The next result concerns weak convergence. In the case of standard
Fej\'er monotonicity \eqref{e:fejer}, it appears in 
\cite[Lemma~6]{Bro67b} and, in the case of quasi-Fej\'er 
monotonicity \eqref{e:qfejer}, it appears in 
\cite[Proposition~1.3]{Albe98}.

\begin{theorem}
\label{t:1} 
Let $\alpha\in\RPP$, let $\phi\colon\RP\to\RP$ be strictly 
increasing and such that $\lim_{t\to\pinf}\phi(t)=\pinf$, 
let $(W_n)_{n\in\NN}$ and $W$ be operators in $\BP_{\alpha}(\HH)$
such that $W_n\to W$ pointwise, let $C$ be a nonempty subset of 
$\HH$, and let $(x_n)_{n\in\NN}$ be a sequence in $\HH$ such 
that \eqref{e:vmqf1} is satisfied. Then 
$(x_n)_{n\in\NN}$ converges weakly to a point in $C$ 
if and only if every weak sequential cluster point of 
$(x_n)_{n\in\NN}$ is in $C$.
\end{theorem}
\begin{proof}
Necessity is clear. To show sufficiency, suppose that every weak 
sequential cluster point of $(x_n)_{n\in\NN}$ is in $C$, and let 
$x$ and $y$ be two such points,
say $x_{k_n}\weakly x$ and $x_{l_n}\weakly y$. Then it follows from
Proposition~\ref{p:1}\ref{p:1i} that $(\|x_n-x\|_{W_n})_{n\in\NN}$ 
and $(\|x_n-y\|_{W_n})_{n\in\NN}$ converge. Moreover,
$\|x\|_{W_n}^2=\scal{W_nx}{x}\to\scal{Wx}{x}$ and, likewise,
$\|y\|_{W_n}^2\to\scal{Wy}{y}$. Therefore, since
\begin{equation}
(\forall n\in\NN)\quad\scal{W_nx_n}{x-y} 
=\frac12\big(\|x_n-y\|_{W_n}^2-\|x_n-x\|_{W_n}^2
+\|x\|_{W_n}^2-\|y\|_{W_n}^2\big),
\end{equation}
the sequence $(\scal{W_nx_n}{x-y})_{n\in\NN}$ converges, say 
$\scal{W_n x_n}{x-y}\to\lambda\in\RR$, which implies that
\begin{equation}
\label{e:elnido2012-03-07b}
\scal{x_n}{W_n(x-y)}\to\lambda\in\RR.
\end{equation}
However, since $x_{k_n}\weakly x$ and $W_{k_n}(x-y)\to W(x-y)$, 
it follows from \eqref{e:elnido2012-03-07b} and 
\cite[Lemma~2.41(iii)]{Livre1} that $\scal{x}{W(x-y)}=\lambda$. 
Likewise, passing to the limit along the subsequence 
$(x_{l_n})_{n\in\NN}$ in \eqref{e:elnido2012-03-07b} yields
$\scal{y}{W(x-y)}=\lambda$. Thus,
\begin{equation}
\label{e:elnido2012-03-08b}
0=\scal{x}{W(x-y)}-\scal{y}{W(x-y)} 
=\scal{x-y}{W(x-y)}\geq\alpha\|x-y\|^2.
\end{equation}
This shows that $x=y$. Upon invoking 
Proposition~\ref{p:1}\ref{p:1ii} and
\cite[Lemma~2.38]{Livre1}, 
we conclude that $x_n\weakly x$.
\end{proof}

Lemma~\ref{l:ES175} provides instances in which the 
conditions imposed on $(W_n)_{n\in\NN}$ in 
Theorem~\ref{t:1} are satisfied. Next, we present
a characterization of strong convergence which can be found 
in \cite[Theorem~3.11]{Else01} in the special case of 
quasi-Fej\'er monotonicity \eqref{e:qfejer}.

\begin{proposition}
\label{p:monodc} 
Let $\alpha\in\RPP$, let $\chi\in\left[1,\pinf\right[$, and let 
$\phi\colon\RP\to\RP$ be an increasing upper semicontinuous 
function vanishing only at $0$ and such that 
\begin{equation}
\label{e:quaso}
\big(\forall(\xi_1,\xi_2)\in\left[0,\pinf\right[^{2}\big)\quad 
\phi(\xi_1+\xi_2)\leq\chi\big(\phi(\xi_1)+\phi(\xi_2)\big).
\end{equation}
Let $(W_n)_{n\in\NN}$ be a sequence in $\BP_{\alpha}(\HH)$ such 
that $\mu=\sup_{n\in\NN}\|W_n\|<\pinf$, let $C$ be a nonempty 
closed subset of $\HH$, and let $(x_n)_{n\in\NN}$ be a sequence 
in $\HH$ such that \eqref{e:vmqf2} is satisfied. Then 
$(x_n)_{n\in\NN}$ converges strongly to a point in $C$ if 
and only if $\varliminf d_C(x_n)=0$.
\end{proposition}
\begin{proof} 
Necessity is clear. For sufficiency, suppose that
$\varliminf d_C(x_n)=0$ and set $(\forall n\in\NN)$ 
$\xi_n=\inf_{z\in C}\|x_n-z\|_{W_n}$. For every $n\in\NN$,
let $(z_{n,k})_{k\in\NN}$ be a sequence in $C$ such that 
$\|x_n-z_{n,k}\|_{W_n}\to\xi_n$. Then, since $\phi$ is increasing,
\eqref{e:vmqf2} yields
\begin{equation}
(\forall n\in\NN)(\forall k\in\NN)\quad\phi(\xi_{n+1}) 
\leq\phi(\|x_{n+1}-z_{n,k}\|_{W_{n+1}})
\leq(1+\eta_n)\phi(\|x_n-z_{n,k}\|_{W_n})+\varepsilon_n.
\end{equation}
Hence, it follows from the upper semicontinuity of $\phi$ that 
\begin{align}
(\forall n\in\NN)\quad\phi(\xi_{n+1}) 
&\leq(1+\eta_n)\varlimsup_{k\to\pinf}\phi(\|x_n-z_{n,k}\|_{W_n})
+\varepsilon_n\nonumber\\
&\leq(1+\eta_n)\phi(\xi_n)+\varepsilon_n.
\end{align}
Therefore, by Lemma~\ref{l:7},
\begin{equation}
\label{e:elnido2012-03-09a}
\big(\phi(\xi_n)\big)_{n\in\NN}\quad\text{converges}.
\end{equation}
Moreover, since
\begin{equation}
\label{e:kj1}
(\forall n\in\NN)(\forall m\in\NN)(\forall x\in\HH)\quad
\alpha\|x_n-x\|^2\leq\|x_n-x\|_{W_m}^{2}\leq 
\mu\|x_n-x\|^2,
\end{equation}
we have
\begin{equation}
\label{e:kj2}
(\forall n\in\NN)\quad\sqrt{\alpha}d_{C}(x_n)\leq\xi_n\leq 
\sqrt{\mu}d_{C}(x_n).
\end{equation}
Consequently, since $\varliminf d_C(x_n)=0$, we derive from 
\eqref{e:kj2} that $\varliminf\xi_n=0$. Let us extract a subsequence
$(\xi_{k_n})_{n\in\NN}$ such that $\xi_{k_n}\to 0$.
Since $\phi$ is upper semicontinuous, we have
$0\leq\varliminf\phi(\xi_{k_n})\leq\varlimsup\phi(\xi_{k_n})\leq
\phi(0)=0$. In view of \eqref{e:elnido2012-03-09a}, we therefore 
obtain $\phi(\xi_n)\to 0$ and, in turn, $\xi_n\to 0$.
Hence, we deduce from \eqref{e:kj2} that 
\begin{equation}
\label{e:dhs}
d_C(x_n)\to 0. 
\end{equation}
Next, let $N$ be the smallest integer such that $N>\sqrt{\mu}$,
and set $\rho=\chi^{N-1}+\sum_{k=1}^{N-1}\chi^k$ if $N>1$; 
$\rho=1$ if $N=1$. Moreover, let $x\in C$ and let $m$ and $n$ 
be strictly positive integers. Using \eqref{e:kj1}, the 
monotonicity of $\phi$, and \eqref{e:quaso}, we obtain
\begin{equation}
\label{e:kj1s}
\phi\big(\|x_n-x\|_{W_m}\big)
\leq\phi\big(\sqrt{\mu}\|x_n-x\|\big)
\leq\phi\big(N\|x_n-x\|\big)
\leq\rho\phi\big(\|x_n-x\|\big).
\end{equation}
Now set $\tau=\prod_{k\in\NN}(1+\eta_k)$. Then $\tau<\pinf$
\cite[Theorem~3.7.3]{Knop56} and we derive from \eqref{e:quaso}, 
\eqref{e:vmqf2}, and \eqref{e:kj1s} that
\begin{align}
\chi^{-1}\phi\big(\|x_{n+m}-x_n\|_{W_{n+m}}\big)
&\leq\chi^{-1}\big(\phi(\|x_{n+m}-x\|_{W_{n+m}}+
\|x_n-x\|_{W_{n+m}}\big)\nonumber\\
&\leq\phi\big(\|x_{n+m}-x\|_{W_{m+n}}\big)+
\phi\big(\|x_n-x\|_{W_{m+n}}\big)\nonumber\\
&\leq\tau\bigg(\phi\big(\|x_n-x\|_{W_n}\big)+\sum_{k=n}^{n+m-1}
\varepsilon_k\bigg)+\phi\big(\|x_n-x\|_{W_{m+n}}\big)\nonumber\\
&\leq\rho(1+\tau)\phi\big(\|x_n-x\|\big)
+\tau\sum_{k\geq n}\varepsilon_k.
\end{align}
Therefore, upon taking the infimum over $x\in C$, we obtain
by upper semicontinuity of $\phi$
\begin{equation}
\label{e:khiemton}
\phi\big(\|x_{n+m}-x_n\|_{W_{n+m}}\big)
\leq\chi\rho(1+\tau)\phi\big(d_C(x_n)\big)
+\chi\tau\sum_{k\geq n}\varepsilon_k.
\end{equation}
Hence, appealing to \eqref{e:dhs} and the summability of
$(\varepsilon_k)_{k\in\NN}$, we deduce from \eqref{e:khiemton}
that, as $n\to\pinf$, 
$\phi(\|x_{n+m}-x_n\|_{W_{n+m}})\to 0$ and, hence, 
$\alpha\|x_{n+m}-x_n\|^2\leq\|x_{n+m}-x_n\|^2_{W_{n+m}}\to 0$.
Thus, $(x_n)_{n\in\NN}$ is a Cauchy sequence in $\HH$ and 
there exists $\overline{x}\in\HH$ such that $x_n\to\overline{x}$.
By continuity of $d_C$ and \eqref{e:dhs}, we obtain
$d_C(\overline{x})=0$ and, since $C$ is closed, 
$\overline{x}\in C$. 
\end{proof}

\section{The quadratic case}
\label{sec:4}

In this section, we focus on the important case when 
$\phi=|\cdot|^2$ in Definition~\ref{d:vmqf}. Our first 
result states that variable metric quasi-Fej\'er monotonicity 
``spreads'' to the convex hull of the target set.

\begin{proposition}
\label{p:jpa} 
Let $\alpha\in\RPP$, let $(\eta_n)_{n\in\NN}$ be a sequence 
in $\ell_+^1(\NN)$, let $(W_n)_{n\in\NN}$ be a 
sequence in $\BP_{\alpha}(\HH)$ such that 
\begin{equation}
\label{e:ronny-J}
\mu=\sup_{n\in\NN}\|W_n\|<\pinf\quad\text{and}\quad
(\forall n\in\NN)\quad (1+\eta_n)W_n\succcurlyeq W_{n+1}.
\end{equation}
Let $C$ be a nonempty subset of $\HH$ and let $(x_n)_{n\in\NN}$ be 
a sequence in $\HH$ such that  
\begin{multline}
\label{e:2012-04-24'''}
\big(\exi(\eta_n)_{n\in\NN}\in\ell_+^1(\NN)\big)
\big(\forall z\in C\big)\big(\exi(\varepsilon_n)_{n\in\NN}\in
\ell_+^1(\NN)\big)(\forall n\in\NN)\\
\|x_{n+1}-z\|_{W_{n+1}}^2\leq (1+\eta_n)
\|x_n-z\|_{W_n}^2+\varepsilon_n.
\end{multline}
Then the following hold.
\begin{enumerate}
\item 
\label{p:jpai}
$(x_n)_{n\in\NN}$ is $|\cdot|^2$-quasi-Fej\'er monotone with 
respect to $\conv C$ relative to $(W_n)_{n\in\NN}$.
\item
\label{p:jpaii} 
For every $y\in\cconv C$, $(\|x_n-y\|_{W_n})_{n\in\NN}$ 
converges.
\end{enumerate}
\end{proposition}
\begin{proof}
Let us fix $z\in\conv C$. There exist finite sets 
$\{z_i\}_{i\in I}\subset C$ and  
$\{\lambda_i\}_{i\in I}\subset\zerounf$ such that 
\begin{equation}
\label{e:fe}
\sum_{i\in I}\lambda_i=1\quad\text{and}\quad
z=\sum_{i\in I}\lambda_iz_i.
\end{equation}
For every $i\in I$, it follows from \eqref{e:2012-04-24'''} that 
there exists a sequence 
$(\varepsilon_{i,n})_{n\in\NN}\in\ell_+^1(\NN)$ such that 
\begin{equation} 
\label{e:fe1}
(\forall n\in\NN)\quad
\|x_{n+1}-z_i\|_{W_{n+1}}^2 
\leq (1+\eta_n)\|x_n-z_i\|_{W_n}^2+\varepsilon_{i,n}.
\end{equation}
Now set 
\begin{equation}
(\forall n\in\NN)\quad
\begin{cases}
\alpha_n=\Frac{1}{2}\sum_{i\in I}\sum_{j\in I}\lambda_i\lambda_j
\|z_i-z_j\|_{W_n}^2\\
\varepsilon_n=(1+\eta_n)\alpha_n-\alpha_{n+1}+
\max\{\varepsilon_{1,n},\ldots,\varepsilon_{m,n}\}.
\end{cases}
\end{equation}
Then $(\max\{\varepsilon_{1,n},\ldots,\varepsilon_{m,n}\})_{n\in\NN}
\in\ell_+^1(\NN)$ and, by \eqref{e:ronny-J}, $(\forall n\in\NN)$
$(1+\eta_n)\alpha_n\geq\alpha_{n+1}$. Hence, Lemma~\ref{l:7}
asserts that $(\alpha_n)_{n\in\NN}$ converges, which implies that
$(\varepsilon_n)_{n\in\NN}\in \ell_+^1(\NN)$. 

\ref{p:jpai}:
Using \eqref{e:fe}, \cite[Lemma~2.13(ii)]{Livre1}, and 
\eqref{e:fe1}, we obtain
\begin{align}
(\forall n\in\NN)\quad\|x_{n+1}-z\|_{W_{n+1}}^2 
&=\sum_{i\in I}\lambda_i\|x_{n+1}-z_i\|_{W_{n+1}}^2
-\alpha_{n+1}\nonumber\\
&\leq (1+\eta_n)\sum_{i\in I}\lambda_i \|x_{n}-z_i\|_{W_{n}}^2 
-\alpha_{n+1}+\max\{\varepsilon_{1,n},\ldots,\varepsilon_{m,n}\} 
\nonumber\\
&=(1+\eta_n)\|x_{n}-z\|_{W_{n}}^2 
+(1+\eta_n)\alpha_n-\alpha_{n+1}
+\max\{\varepsilon_{1,n},\ldots,\varepsilon_{m,n}\} \nonumber\\
&=(1+\eta_n)\|x_n-z\|_{W_n}^2+ \varepsilon_n.
\end{align}

\ref{p:jpaii}: 
It follows from \cite[Lemma 2.13(ii)]{Livre1} that
\begin{equation}
(\forall n\in\NN)\quad
\|x_n-z\|_{W_n}^2 
=\sum_{i\in I}\lambda_i\|x_n-z_i\|_{W_n}^2-\alpha_n.
\end{equation}
However, $(\alpha_n)_{n\in\NN}$ converges and, for every $i\in I$, 
Proposition~\ref{p:1}\ref{p:1i} asserts that 
$(\|x_n-z_i\|_{W_n})_{n\in\NN}$ converges. Hence,
$(\|x_n-z\|_{W_n})_{n\in\NN}$ converges. Now let $y\in\cconv C$. 
Then there exists a sequence $(y_k)_{k\in\NN}$ in $\conv C$ 
such that $y_k\to y$. It follows from \ref{p:jpai} and 
Proposition~\ref{p:1}\ref{p:1i} that, for every $k\in\NN$, 
$(\|x_n-y_k\|_{W_n})_{n\in\NN}$ converges. Moreover, we have 
\begin{align}
(\forall k\in\NN)(\forall n\in\NN)\quad -\sqrt{\mu} 
\|y_k-y\|&\leq-\|y_k-y\|_{W_n}\nonumber\\
&\leq\|x_n-y\|_{W_n}-\|x_n-y_k\|_{W_n}\nonumber\\
&\leq\|y_k-y\|_{W_n}\nonumber\\
&\leq\sqrt{\mu}\|y_k-y\|.
\end{align}
Consequently,
\begin{align}
(\forall k\in\NN)\quad -\sqrt{\mu}\|y_k-y\|
&\leq\varliminf\|x_n-y\|_{W_n}-\lim\|x_n-y_k\|_{W_n}\nonumber\\
&\leq\varlimsup\|x_n-y\|_{W_n}-\lim\|x_n-y_k\|_{W_n}\nonumber\\
&\leq\sqrt{\mu}\|y_k-y\|.
\end{align}
Taking the limit as $k\to\pinf$ yields
$\lim_{n\to\pinf}\|x_n-y\|_{W_n}=\lim_{k\to\pinf}\lim_{n\to\pinf}
\|x_n-y_k\|_{W_n}$.
\end{proof}

Standard Fej\'er monotone sequences may fail to converge weakly 
and, even when they converge weakly, strong convergence may fail 
\cite{Else01,Hund04}. However, if the target set $C$ is closed 
and convex in \eqref{e:fejer}, the projected sequence 
$(P_Cx_n)_{n\in\NN}$ converges strongly; see 
\cite[Theorem~2.16(iv)]{Baus96} and \cite[Remark~1]{Reic79}.
This property, which remains true in the quasi-Fej\'erian 
case \cite[Proposition~3.6(iv)]{Else01}, is extended below.

\begin{proposition}
\label{p:guad2012}
Let $\alpha\in\RPP$, let $(\eta_n)_{n\in\NN}$ be a sequence in 
$\ell_+^1(\NN)$, let $(W_n)_{n\in\NN}$ be a uniformly bounded 
sequence in $\BP_{\alpha}(\HH)$, let $C$ be a nonempty closed 
convex subset of $\HH$, and let $(x_n)_{n\in\NN}$ be a sequence 
in $\HH$ such that  
\begin{multline}
\label{e:2012-05-24}
\big(\exi(\varepsilon_n)_{n\in\NN}\in\ell_+^1(\NN)\big)
\big(\exi(\eta_n)_{n\in\NN}\in\ell_+^1(\NN)\big)
(\forall z\in C)(\forall n\in\NN)\\
\|x_{n+1}-z\|^2_{W_{n+1}}\leq
(1+\eta_n)\|x_n-z\|^2_{W_n}+\varepsilon_n.
\end{multline}
Then $(P_C^{W_n}x_n)_{n\in\NN}$ converges strongly.
\end{proposition}
\begin{proof} 
Set $(\forall n\in\NN)$ $z_n=P_C^{W_n}x_n$. For every 
$(m,n)\in\NN^2$, since $z_n\in C$ and 
$z_{m+n}=P_{C}^{W_{n+m}}x_{n+m}$, the well-known convex projection 
theorem \cite[Theorem~3.14]{Livre1} yields
\begin{equation}
\scal{z_n-z_{n+m}}{x_{n+m}-z_{n+m}}_{W_{n+m}}\leq 0,
\end{equation}
which implies that 
\begin{align}
\scal{z_n-x_{n+m}}{x_{n+m}-z_{n+m}}_{W_{n+m}} 
&=\scal{z_n-z_{n+m}}{x_{n+m}-z_{n+m}}_{W_{n+m}} 
-\|x_{n+m}-z_{n+m}\|^{2}_{W_{n+m}}\nonumber\\
&\leq-\|x_{n+m}-z_{n+m}\|^{2}_{W_{n+m}}.
\end{align}
Therefore, for every $(m,n)\in\NN^2$, 
\begin{align}
\label{e:shadow1}
\|z_n-z_{n+m}\|_{W_{n+m}}^2 
&=\|z_n-x_{n+m}\|_{W_{n+m}}^2
+2\scal{z_n-x_{n+m}}{x_{n+m}-z_{n+m}}_{W_{n+m}}\nonumber\\
&\quad\;+\|x_{n+m}-z_{n+m}\|_{W_{n+m}}^2\nonumber\\
&\leq\|z_n-x_{n+m}\|_{W_{n+m}}^2-
\|x_{n+m}-z_{n+m}\|^{2}_{W_{n+m}} .
\end{align}
Now fix $z\in C$, and set $\mu=\sup_{n\in\NN}\|W_n\|$ and
$\rho=\sup_{n\in\NN}\|x_n-z\|^2_{W_n}$. Then $\mu<\pinf$ and,
in view of Proposition~\ref{p:1}\ref{p:1i}, $\rho<\pinf$.
It follows from \eqref{e:2012-05-24} that, 
for every $n\in\NN$ and every $m\in\NN\smallsetminus\{0\}$,
since $P_C^{W_n}$ is nonexpansive with respect to $\|\cdot\|_{W_n}$
\cite[Proposition~4.8]{Livre1}, we have
\begin{align}
\label{e:2012-05-28a}
\|x_{n+m}-z_n\|^2_{W_{n+m}}
&\leq\|x_n-z_n\|^2_{W_n}+\sum_{k=n}^{n+m-1}\big(\eta_k
\|x_k-z_n\|^2_{W_k}+\varepsilon_k\big)\nonumber\\
&\leq\|x_n-z_n\|^2_{W_n}+\sum_{k=n}^{n+m-1}\Big(2\eta_k
\big(\|x_k-z\|^2_{W_k}+\|z_n-z\|^2_{W_k}\big)+
\varepsilon_k\Big)\nonumber\\
&\leq\|x_n-z_n\|^2_{W_n}+\sum_{k=n}^{n+m-1}\Big(2\eta_k\Big(\rho+
\frac{\mu}{\alpha}\|P_C^{W_n}x_n-P_C^{W_n}z\|^2_{W_n}\Big)+
\varepsilon_k\Big)\nonumber\\
&\leq\|x_n-z_n\|^2_{W_n}+\sum_{k=n}^{n+m-1}\Big(2\eta_k
\Big(\rho+\frac{\mu}{\alpha}\|x_n-z\|^2_{W_n}\Big)+
\varepsilon_k\Big)\nonumber\\
&\leq\|x_n-z_n\|^2_{W_n}+\sum_{k=n}^{n+m-1}\Big(2\rho\eta_k
\Big(1+\frac{\mu}{\alpha}\Big)+\varepsilon_k\Big).
\end{align}
Combining \eqref{e:shadow1} and \eqref{e:2012-05-28a}, we obtain
that for every $n\in\NN$ and every $m\in\NN\smallsetminus\{0\}$,
\begin{align}
\label{e:oitroi}
\alpha\|z_{n+m}-z_n\|^2
&\leq\|z_{n+m}-z_n\|_{W_{n+m}}^2\nonumber\\
&\leq\|x_n-z_n\|_{W_n}^2-\|x_{n+m}-z_{n+m}\|^{2}_{W_{n+m}}
+\sum_{k\geq n}\Big(2\rho\eta_k
\Big(1+\frac{\mu}{\alpha}\Big)+\varepsilon_k\Big).
\end{align}
On the other hand, \eqref{e:2012-05-24} yields
\begin{align}
(\forall n\in\NN)\quad\|x_{n+1}-z_{n+1}\|_{W_{n+1}}^2
&\leq\|x_{n+1}-z_{n}\|_{W_{n+1}}^2\nonumber\\
&\leq(1+\eta_n)\|x_n-z_n\|_{W_n}^2+\varepsilon_n,
\end{align}
which, by Lemma~\ref{l:7}, implies that 
$(\|x_n-z_n\|_{W_n})_{n\in\NN}$ converges.
Consequently, since $(\eta_k)_{k\in\NN}$ and 
$(\varepsilon_k)_{k\in\NN}$ are in $\ell_+^1(\NN)$, we derive 
from \eqref{e:oitroi} that  $(z_n)_{n\in\NN}$ is a Cauchy 
sequence and hence that it converges strongly.
\end{proof}

In the case of classical Fej\'er monotone sequences, it has been 
known since \cite{Raik69} that strong convergence is achieved 
when the interior of the target set is nonempty (see also
\cite[Proposition~3.10]{Else01} for the case of quasi-Fej\'er 
monotonicity). The following result extends this fact in the 
context of variable metric quasi-Fej\'er sequences.

\begin{proposition}
\label{p:2012-03-26}
Let $\alpha\in\RPP$, let $(\nu_n)_{n\in\NN}\in\ell_+^1(\NN)$, and
let $(W_n)_{n\in\NN}$ be a sequence in $\BP_{\alpha}(\HH)$
such that
\begin{equation}
\label{e:nangam}
\mu=\sup_{n\in\NN}\|W_n\|<\pinf\quad\text{and}\quad
(\forall n\in\NN)\quad(1+\nu_n)W_{n+1}\succcurlyeq W_n.
\end{equation}
Furthermore, let $C$ be a subset of $\HH$ such that 
$\inte C\neq\emp$, let $z\in C$ and $\rho\in\RPP$ be such that
$B(z;\rho)\subset C$, and let $(x_n)_{n\in\NN}$ be a sequence in 
$\HH$ such that 
\begin{multline}
\label{e:vmqf3}
\big(\exi(\varepsilon_n)_{n\in\NN}\in\ell_+^1(\NN)\big)
\big(\exi(\eta_n)_{n\in\NN}\in\ell_+^1(\NN)\big)
(\forall x\in B(z;\rho))(\forall n\in\NN)\\
\|x_{n+1}-x\|^2_{W_{n+1}}\leq
(1+\eta_n)\|x_n-x\|^2_{W_n}+\varepsilon_n.
\end{multline}
Then $(x_n)_{n\in\NN}$ converges strongly.
\end{proposition}
\begin{proof}
We derive from \eqref{e:nangam} and 
Proposition~\ref{p:1}\ref{p:1ii} that 
\begin{equation}
\label{e:2012-03-29a}
\zeta=\sup_{x\in B(z;\rho)}\sup_{n\in\NN}\|x_n-x\|_{W_n}^2
\leq2\mu\bigg(\sup_{n\in\NN}
\|x_n-z\|^2+\sup_{x\in B(z;\rho)}\|x-z\|^2\bigg)
<\pinf.
\end{equation}
It follows from \eqref{e:vmqf3} and \eqref{e:2012-03-29a} that
\begin{equation}
\label{e:balay2012-03-13ff}
(\forall n\in\NN)(\forall x\in B(z;\rho))\quad
\|x_{n+1}-x\|^2_{W_{n+1}}\leq\|x_n-x\|_{W_n}^2+\xi_n,
\quad\text{where}\quad\xi_n=\zeta\eta_n+\varepsilon_n.
\end{equation}
Now set 
\begin{equation}
\label{e:hoaquasoin}
(\forall n\in\NN)\quad v_n=W_{n+1}(x_{n+1}-z)-W_n(x_n-z),
\end{equation}
and define a sequence $(z_n)_{n\in\NN}$ in $B(z;\rho)$ by
\begin{equation}
\label{e:2012-03-23}
(\forall n\in\NN)\quad z_n=z-\rho u_n,\quad\text{where}\quad u_n=
\begin{cases}
0,&\text{if}\;\;v_n=0;\\
\displaystyle{v_n}/{\|v_n\|},&\text{if}\;\;v_n\neq 0.
\end{cases}
\end{equation}
Then
\begin{equation}
\label{e:2012-03-24}
(\forall n\in\NN)\quad
\begin{cases}
\|x_{n+1}-z_n\|_{W_{n+1}}^2&=\|x_{n+1}-z\|_{W_{n+1}}^2+
2\rho\scal{W_{n+1}(x_{n+1}-z)}{u_n}\\
&\quad\;+\,\rho^2\|u_n\|_{W_{n+1}}^2;\\
\|x_n-z_n\|_{W_n}^2&=\|x_n-z\|_{W_n}^2+
2\rho\scal{W_n(x_n-z)}{u_n}+\rho^2\|u_n\|_{W_n}^2.
\end{cases}
\end{equation}
On the other hand, \eqref{e:balay2012-03-13ff} yields 
$(\forall n\in\NN)$ 
$\|x_{n+1}-z_n\|^2_{W_{n+1}}\leq\|x_n-z_n\|_{W_n}^2+\xi_n$.  
Therefore, it follows from \eqref{e:2012-03-24},
\eqref{e:hoaquasoin}, and \eqref{e:nangam} that
\begin{align}
\label{e:balay2012-03-13fg}
(\forall n\in\NN)\quad\|x_{n+1}-z\|^2_{W_{n+1}}
&\leq\|x_n-z\|_{W_n}^2-2\rho\|v_n\|+\rho^2
\big(\|u_n\|_{W_n}^{2}-\|u_n\|_{W_{n+1}}^{2}\big)+\xi_n
\nonumber\\
&\leq\|x_n-z\|_{W_n}^2-2\rho\|v_n\|+\rho^2\mu\nu_n+\xi_n.
\end{align}
Since $(\rho^2\mu\nu_n+\xi_n)_{n\in\NN}\in\ell_+^1(\NN)$, this
implies that 
\begin{equation}
\label{e:balay2012-03-13fi}
\sum_{n\in\NN}\|w_{n+1}-w_n\|=
\sum_{n\in\NN}\|v_n\|<\pinf,\quad\text{where}\quad
(\forall n\in\NN)\quad w_n=W_n(x_n-z). 
\end{equation}
Hence, $(w_n)_{n\in\NN}$ is a Cauchy sequence in $\HH$ and,
therefore, there exists $w\in\HH$ such that $w_n\to w$. On the
other hand, we deduce from \eqref{e:nangam} and
Lemma~\ref{l:ES175}\ref{l:ES175ii} that there exists 
$W\in\BP_{\alpha}(\HH)$ such that $W_n\to W$.
Now set $x=z+W^{-1}w$. Then, since $(W_n)_{n\in\NN}$ lies
in $\BP_\alpha(\HH)$, it follows from Cauchy-Schwarz that
\begin{equation}
\label{e:kjMMXII}
\alpha\|x_n-x\|
\leq\|W_nx_n-W_nx\|
=\|w_n-W_nW^{-1}w\|
\leq\|w_n-w\|+\|w-W_nW^{-1}w\|
\to 0,
\end{equation}
which concludes the proof.
\end{proof}

\section{Application to convex feasibility}
\label{sec:5}

We illustrate our results through an application to the convex 
feasibility problem, i.e., the generic problem of finding a
common point of a family of closed convex sets. As in
\cite{Moor01}, given $\alpha\in\RPP$ and $W\in\BP_{\alpha}(\HH)$, 
we say that an operator $T\colon\HH\to\HH$ with fixed point set 
$\Fix T$ belongs to $\mathfrak{T}(W)$ if
\begin{equation}
\label{e:md8}
(\forall x\in\HH)(\forall y\in\Fix T)\quad
\scal{y-Tx}{x-Tx}_W\leq 0.
\end{equation}
If $T\in\mathfrak{T}(W)$, then \cite[Proposition~2.3(ii)]{Else01}
yields
\begin{multline}
\label{e:md9}
(\forall x\in\HH)(\forall y\in\Fix T)(\forall\lambda\in[0,2])\quad
\|(\Id+\lambda(T-\Id))x-y\|_W^2\\
\leq\|x-y\|_W^2-\lambda(2-\lambda)\|Tx-x\|_W^2.
\end{multline}
The usefulness of the class $\mathfrak{T}(W)$ stems from the fact 
that it contains many of the operators commonly encountered in 
nonlinear analysis: firmly nonexpansive operators (in particular 
resolvents of maximally monotone operators and proximity 
operators of proper lower semicontinuous convex functions), 
subgradient projection operators, projection operators, 
averaged quasi-nonexpansive operators, and several combinations 
thereof \cite{Moor01,Kruk06,Else01}.

\begin{theorem}
\label{t:2}
Let $\alpha\in\RPP$, let $(C_i)_{i\in I}$ be a finite or 
countably infinite family of closed convex subsets of $\HH$ such 
that $C=\bigcap_{i\in I}C_i\neq\emp$, let 
$(a_n)_{n\in\NN}$ be a sequence in $\HH$ such that 
$\sum_{n\in\NN}\|a_n\|<\pinf$, 
let $(\eta_n)_{n\in\NN}$ be a sequence in 
$\ell_{+}^1(\NN)$, and let $(W_n)_{n\in\NN}$ be a sequence in 
$\BP_{\alpha}(\HH)$ such that  
\begin{equation}
\label{e:guad2012a}
\mu=\sup_{n\in\NN}\|W_n\|<\pinf
\quad\text{and}\quad
(\forall n\in\NN)\quad(1+\eta_n)W_n\succcurlyeq W_{n+1}.
\end{equation}
Let $\operatorname{i}\colon\NN\to I$ be such that 
\begin{equation}
\label{e:browder}
(\forall j\in I)(\exi M_j\in\NN\smallsetminus\{0\})(\forall n\in\NN)
\quad j\in\{\operatorname{i}(n),\ldots,\operatorname{i}(n+M_j-1)\}.
\end{equation}
For every $i\in I$, let $(T_{i,n})_{n\in\NN}$ be a sequence of 
operators such that 
\begin{equation}
\label{e:md7}
(\forall n\in\NN)\quad {T}_{i,n}\in\mathfrak{T}(W_n)\quad
\text{and}\quad\Fix{T}_{i,n}=C_i.
\end{equation}
Fix $\varepsilon\in\left]0,1\right[$ and $x_0\in\HH$, let 
$(\lambda_n)_{n\in\NN}$ be a sequence in 
$[\varepsilon,2-\varepsilon]$, and set
\begin{equation}
\label{e:guad2012b}
(\forall n\in\NN)\quad x_{n+1}=
x_n+\lambda_n\big(T_{\operatorname{i}(n),n}x_n+a_n-x_n\big).
\end{equation}
Suppose that, for every strictly increasing sequence
$(p_n)_{n\in\NN}$ in $\NN$, every $x\in\HH$, and every $j\in I$,
\begin{equation}
\label{e:mdX}
\begin{cases}
x_{p_n}\weakly x\\
T_{j,p_n}x_{p_n}-x_{p_n}\to 0\\
(\forall n\in\NN)\;\;j=\operatorname{i}(p_n)
\end{cases}
\quad\Rightarrow\quad x\in C_j.
\end{equation}
Then the following hold for some $\overline{x}\in C$.
\begin{enumerate}
\item 
\label{t:2i}
$x_n\weakly\overline{x}$.
\item 
\label{t:2ii}
Suppose that $\inte C\neq\emp$ and that there exists 
$(\nu_n)_{n\in\NN}\in\ell_{+}^1(\NN)$ such that $(\forall n\in\NN)$
$(1+\nu_n)W_{n+1} \succcurlyeq W_n$. Then $x_n\to\overline{x}$.
\item 
\label{t:2iii}
Suppose that $\varliminf d_C(x_n)=0$. Then $x_n\to\overline{x}$.
\item 
\label{t:2iv}
Suppose that there exists an index $j\in I$ of demicompact 
regularity: for every strictly increasing sequence 
$(p_n)_{n\in\NN}$ in $\NN$,
\begin{equation}
\label{e:j-pp}
\begin{cases}
\sup_{n\in\NN}\|x_{p_n}\|<\pinf\\
T_{j,{p_n}}x_{p_n}-x_{p_n}\to 0\\
(\forall n\in\NN)\;\;j=\operatorname{i}(p_n)
\end{cases}
\quad\Rightarrow\quad (x_{p_n})_{n\in\NN}\;\text{has a strong
sequential cluster point}.
\end{equation}
Then $x_n\to\overline{x}$.
\end{enumerate}
\end{theorem}
\begin{proof}
Fix $z\in C$ and set 
\begin{equation}
\label{e:guad2012m}
(\forall n\in\NN)\quad
y_n=x_n+\lambda_n\big(T_{\operatorname{i}(n),n}x_n -x_n\big).
\end{equation}
Appealing to \eqref{e:md9} and the fact that, by virtue of
\eqref{e:browder}, $z\in\bigcap_{i\in I}C_i=\bigcap_{n\in\NN}\Fix 
T_{\operatorname{i}(n),n}$, we obtain,
\begin{align}
(\forall n\in\NN)\quad
\|y_{n}-z\|_{W_n}^2 &\leq
\|x_n-z\|_{W_n}^2-\lambda_n(2-\lambda_n) 
\|T_{\operatorname{i}(n),n}x_n-x_{n}\|_{W_n}^2\nonumber\\
&\leq\|x_n-z\|_{W_n}^2-\varepsilon^2
\|T_{\operatorname{i}(n),n}x_n-x_{n}\|_{W_n}^2.
\end{align}
Moreover, it follows from \eqref{e:guad2012a} that
\begin{equation}
\label{e:guad2012d}
(\forall n\in\NN)\quad
\|y_{n}-z\|_{W_{n+1}}^2\leq (1+\eta_n)\|y_{n}-z\|_{W_n}^2 .
\end{equation}
Thus,
\begin{align}
(\forall n\in\NN)\quad
\|y_{n}-z\|_{W_{n+1}}^2&\leq(1+\eta_n)\|x_n-z\|_{W_n}^2
-\varepsilon^2 (1+\eta_n)\|T_{\operatorname{i}(n),n}x_n-x_{n}
\|_{W_n}^2\nonumber\\
&\leq (1+\eta_n)\|x_n-z\|_{W_n}^2
-\varepsilon^2\|T_{\operatorname{i}(n),n}x_n-x_{n}\|_{W_n}^2
\label{e:cyc}\\
&\leq (1+\eta_n)\|x_n-z\|_{W_n}^2.
\label{e:md11}
\end{align}
Using \eqref{e:guad2012b}, \eqref{e:guad2012m}, and
\eqref{e:md11},  we get
\begin{align}
\label{e:cycA}
(\forall n\in\NN)\quad
\|x_{n+1}-z\|_{W_{n+1}}
&\leq\|y_n-z\|_{W_{n+1}}+\lambda_{n}\|a_n\|_{W_{n+1}}\nonumber\\
&\leq\sqrt{1+\eta_n}\|x_n-z\|_{W_n}+\sqrt{\mu}\lambda_{n}\|a_n\|
\nonumber\\
&\leq (1+\eta_n)\|x_n-z\|_{W_n}+2\sqrt{\mu}\|a_n\|,
\end{align}
which shows that 
\begin{equation}
\label{e:2012-08-18-md2}
\text{$(x_n)_{n\in\NN}$ satisfies \eqref{e:vmqf2} --
and hence \eqref{e:vmqf1} -- with $\phi=|\cdot|$.}
\end{equation}
It follows from \eqref{e:2012-08-18-md2} and
Proposition~\ref{p:1}\ref{p:1i} that 
$(\|x_n-z\|_{W_n} )_{n\in\NN}$ converges, say 
\begin{equation}
\label{e:guad2012l}
\|x_n-z\|_{W_n}\to\xi\in\RR.
\end{equation}
We therefore derive from \eqref{e:cycA} that
$\|y_n-z\|_{W_{n+1}}\to\xi$ and then from \eqref{e:cyc} that 
\begin{equation}
\label{e:guad2012y}
\alpha\varepsilon^2\|T_{\operatorname{i}(n),n}x_n-x_n\|^2\leq
\varepsilon^2\|T_{\operatorname{i}(n),n}x_n-x_n\|_{W_n}^2
\leq (1+\eta_n)\|x_n-z\|_{W_n}^2 
-\|y_{n}-z\|_{W_{n+1}}^2\to 0.
\end{equation}

\ref{t:2i}:
It follows from \eqref{e:guad2012b} and \eqref{e:guad2012y} that
\begin{align}
\label{e:guad2012c}
\|x_{n+1}-x_n\|
&=\lambda_n\big\|T_{\operatorname{i}(n),n}x_n+a_n-x_n\big\|
\nonumber\\
&\leq 2\big(\|T_{\operatorname{i}(n),n}x_n-x_n\|+
\|a_n\|\big)\nonumber\\
&\leq 2\big(\|T_{\operatorname{i}(n),n}x_n-x_n\|_{W_n}/
\sqrt{\alpha}+\|a_n\|\big)\nonumber\\
&\to 0.
\end{align}
Now, fix $j\in I$ and let $x$ be a weak sequential cluster point 
of $(x_n)_{n\in\NN}$. According to \eqref{e:browder}, there exist 
strictly increasing sequences $(k_n)_{n\in\NN}$ and 
$(p_n)_{n\in\NN}$ in $\NN$ such that $x_{k_n}\weakly x$ and 
\begin{equation}
\label{e:guad2012r}
(\forall n\in\NN)\quad
\begin{cases}
k_n\leq p_n\leq k_n+M_j-1<k_{n+1}\leq p_{n+1},\\
j=\operatorname{i}(p_n). 
\end{cases}
\end{equation}
Therefore, we deduce from \eqref{e:guad2012c} that
\begin{align}
\|x_{p_n}-x_{k_n}\|&\leq\sum_{l=k_n}^{k_n+M_j-2}\|x_{l+1}-x_l\|
\nonumber\\
&\leq(M_j-1)\max_{k_n\leq l\leq k_n+M_j-2}\|x_{l+1}-x_l\|\nonumber\\
&\to 0,
\end{align}
which implies that $x_{p_n}\weakly x$. We also derive
from \eqref{e:guad2012y} and \eqref{e:guad2012r} that 
$T_{j,p_n}x_{p_n}-x_{p_n}=
T_{\operatorname{i}(p_n),p_n}x_{p_n}-x_{p_n}\to 0$. Altogether, 
it follows from \eqref{e:mdX} that $x\in C_j$. Since $j$ was 
arbitrarily chosen in $I$, we obtain $x\in C$ and, in view of 
Lemma~\ref{l:ES175}\ref{l:ES175i} and Theorem~\ref{t:1}, 
we conclude that $x_n\weakly x$.

\ref{t:2ii}: 
Suppose that $z\in\inte C$ and fix $\rho\in\RPP$ such that
$B(z;\rho)\subset C$. Set $\eta=\sup_{n\in\NN}\eta_n$,
$\zeta=\sup_{x\in B(z;\rho)}\sup_{n\in\NN}\|x_n-x\|_{W_n}$, and 
\begin{equation}
\label{e:md12}
(\forall n\in\NN)\quad\varepsilon_n=4\big(\zeta\sqrt{\mu(1+\eta)}
\|a_n\|+\mu\|a_n\|^2\big).
\end{equation}
Then $\eta<\pinf$ and, as in \eqref{e:2012-03-29a}, $\zeta<\pinf$. 
Therefore $(\varepsilon_n)_{n\in\NN}\in\ell_+^1(\NN)$.
Furthermore, we derive from \eqref{e:guad2012b}, 
\eqref{e:guad2012m}, and \eqref{e:md11} that, for every
$x\in B(z;\rho)$ and every $n\in\NN$, 
\begin{align}
\label{e:guad2012j}
\|x_{n+1}-x\|^2_{W_{n+1}}
&\leq\|y_n-x\|^2_{W_{n+1}}+2\lambda_n\|y_n-x\|_{W_{n+1}}\,
\|a_n\|_{W_{n+1}}+\lambda_n^2\|a_n\|^2_{W_{n+1}}\nonumber\\
&\leq(1+\eta_n)\|x_n-x\|^2_{W_n}+4\sqrt{\mu(1+\eta_n)}
\|x_n-x\|_{W_n}\,\|a_n\|+4\mu\|a_n\|^2\nonumber\\
&\leq(1+\eta_n)\|x_n-x\|^2_{W_n}+\varepsilon_n.
\end{align}
Altogether, the assertion follows from \ref{t:2i} and 
Proposition~\ref{p:2012-03-26}.

\ref{t:2iii}: This follows from \eqref{e:2012-08-18-md2},
Proposition~\ref{p:monodc}, and \ref{t:2i}.

\ref{t:2iv}: Let $j\in I$ be an index of demicompact regularity
and let $(p_n)_{n\in\NN}$ be a strictly increasing sequence such
that $(\forall n\in\NN)$ $j=\operatorname{i}(p_n)$. Then 
$(x_{p_n})_{n\in\NN}$ is bounded, while \eqref{e:guad2012y} 
asserts that $T_{j,{p_n}}x_{p_n}-x_{p_n}\to 0$. In turn, 
\eqref{e:j-pp} and \ref{t:2i} imply that 
$x_{p_n}\to\overline{x}\in C$. Therefore
$\varliminf d_C(x_n)\leq\|x_{p_n}-\overline{x}\|\to 0$ and 
\ref{t:2iii} yields the result.
\end{proof}

Condition \eqref{e:browder} first appeared in
\cite[Definition~5]{Bro67b}. Property \eqref{e:mdX} was 
introduced in \cite[Definition~3.7]{Baus96} and property 
\eqref{e:j-pp} in \cite[Definition~6.5]{Else01}. Examples 
of sequences of operators that satisfy \eqref{e:mdX} can 
be found in \cite{Baus96,Kruk06,Else01}. Here is a simple 
application of Theorem~\ref{t:2} to a variable metric 
periodic projection method.

\begin{corollary}
\label{c:2}
Let $\alpha\in\RPP$, let $m$ be a strictly positive integer, 
let $I=\{1,\ldots,m\}$, let $(C_i)_{i\in I}$ be family of 
closed convex subsets of $\HH$ such that 
$C=\bigcap_{i\in I}C_i\neq\emp$, let 
$(a_n)_{n\in\NN}$ be a sequence in $\HH$ such that 
$\sum_{n\in\NN}\|a_n\|<\pinf$,
let $(\eta_n)_{n\in\NN}$ be a sequence in 
$\ell_{+}^1(\NN)$, and let $(W_n)_{n\in\NN}$ be a sequence in 
$\BP_{\alpha}(\HH)$ such that  
$\sup_{n\in\NN}\|W_n\|<\pinf$ and 
$(\forall n\in\NN)$ $(1+\eta_n)W_n\succcurlyeq W_{n+1}$.
Fix $\varepsilon\in\left]0,1\right[$ and $x_0\in\HH$, let 
$(\lambda_n)_{n\in\NN}$ be a sequence in 
$[\varepsilon,2-\varepsilon]$, and set
\begin{equation}
\label{e:guad2012s}
(\forall n\in\NN)\quad 
x_{n+1}=x_n+\lambda_n\Big(
P_{C_{1+\operatorname{rem}(n,m)}}^{W_n}x_n+a_n-x_n\Big),
\end{equation}
where $\operatorname{rem}(\cdot,m)$ is the 
remainder function of the division by $m$. 
Then the following hold for some $\overline{x}\in C$.
\begin{enumerate}
\item 
\label{c:2i}
$x_n\weakly\overline{x}$.
\item 
\label{c:2ii}
Suppose that $\inte C\neq\emp$ and that there exists 
$(\nu_n)_{n\in\NN}\in\ell_{+}^1(\NN)$ such that $(\forall n\in\NN)$
$(1+\nu_n)W_{n+1} \succcurlyeq W_n$. Then $x_n\to\overline{x}$.
\item 
\label{c:2iii}
Suppose that there exists $j\in I$ such that $C_j$ is boundedly
compact, i.e., its intersection with every closed ball of $\HH$ is
compact. Then $x_n\to\overline{x}$.
\end{enumerate}
\end{corollary}
\begin{proof}
The function $\operatorname{i}\colon\NN\to I\colon n\mapsto
1+\operatorname{rem}(n,m)$ satisfies \eqref{e:browder}
with $(\forall j\in I)$ $M_j=m$. Now, 
set $(\forall i\in I)(\forall n\in\NN)$ $T_{i,n}=P_{C_i}^{W_n}$. 
Then $(\forall i\in I)(\forall n\in\NN)$
$T_{i,n}\in\mathfrak{T}(W_n)$ and $\Fix T_{i,n}=C_i$. Hence, 
\eqref{e:guad2012s} is a special case of \eqref{e:guad2012b}. 

\ref{c:2i}--\ref{c:2ii}:
Fix $j\in I$ and let $(x_{p_n})_{n\in\NN}$ be a weakly
convergent subsequence of $(x_n)_{n\in\NN}$, say 
$x_{p_n}\weakly x$, such that
$T_{j,p_n}x_{p_n}-x_{p_n}\to 0$ and $(\forall n\in\NN)$ 
$j=\operatorname{i}(p_n)$. Then 
$C_j\ni P_{C_j}^{W_{p_n}}x_{p_n}=T_{j,p_n}x_{p_n}\weakly x$ and,
since $C_j$ is weakly closed \cite[Theorem~3.32]{Livre1}, we have
$x\in C_j$. This shows that \eqref{e:mdX} holds. Altogether, 
the claims follow from Theorem~\ref{t:2}\ref{t:2i}--\ref{t:2ii}.

\ref{c:2iii}: Let $(p_n)_{n\in\NN}$ be a strictly increasing
sequence in $\NN$ such that 
$P^{W_{p_n}}_{C_j}x_{p_n}-x_{p_n}=T_{j,p_n}x_{p_n}-x_{p_n}\to 0$
and $(\forall n\in\NN)$ $j=\operatorname{i}(p_n)$. 
Then 
\begin{equation}
\label{e:athens2012-08-22x}
\|P_{C_j}x_{p_n}-x_{p_n}\|\leq
\|P^{W_{p_n}}_{C_j}x_{p_n}-x_{p_n}\|\to 0.
\end{equation}
On the other hand, since $(x_{p_n})_{n\in\NN}$ is bounded and 
$P_{C_j}$ is nonexpansive, $(P_{C_j}x_{p_n})_{n\in\NN}$ is 
a bounded sequence in the boundedly compact set $C_j$. Hence, 
$(P_{C_j}x_{p_n})_{n\in\NN}$ admits a strong sequential 
cluster point and so does $(x_{p_n})_{n\in\NN}$ since 
$P_{C_j}x_{p_n}-x_{p_n}\to 0$. Thus, $j\in I$ is an index of 
demicompact regularity and the claim therefore follows from 
Theorem~\ref{t:2}\ref{t:2iv}.
\end{proof}

\begin{remark}
In the special case when, for every $n\in\NN$, $W_n=\Id$ and
$\eta_n=0$, Corollary~\ref{c:2}\ref{c:2i} was established in 
\cite{Breg65} (with $(\forall n\in\NN)$ $\lambda_n=1)$, 
and Corollary~\ref{c:2}\ref{c:2ii} in \cite{Gubi67}.
\end{remark}

Next is an application of Corollary~\ref{c:2} to the problem 
of solving linear inequalities. In Euclidean spaces, the use 
of periodic projection methods to solve this problem goes back 
to \cite{Motz54}.

\begin{example}
\label{ex:2012-08-19-md}
Let $\alpha\in\RPP$, let $m$ be a strictly positive integer, 
let $I=\{1,\ldots,m\}$, let $(\eta_i)_{i\in I}$ be real numbers, 
and suppose that $(u_i)_{i\in I}$ are nonzero vectors in $\HH$ 
such that
\begin{equation}
\label{e:2012-08-19-md1}
C=\menge{x\in\HH}{(\forall i\in I)\;\;
\scal{x}{u_i}\leq\eta_i}\neq\emp.
\end{equation}
Let $(\eta_n)_{n\in\NN}$ be a sequence in 
$\ell_{+}^1(\NN)$, and let $(W_n)_{n\in\NN}$ be a sequence in 
$\BP_{\alpha}(\HH)$ such that $\sup_{n\in\NN}\|W_n\|<\pinf$ and 
$(\forall n\in\NN)$ $(1+\eta_n)W_n\succcurlyeq W_{n+1}$.
Fix $\varepsilon\in\left]0,1\right[$ and $x_0\in\HH$, let 
$(\lambda_n)_{n\in\NN}$ be a sequence in 
$[\varepsilon,2-\varepsilon]$, and set
\begin{equation}
\label{e:2012-08-19-md2}
(\forall n\in\NN)\quad 
\left\lfloor
\begin{array}{l}
\operatorname{i}(n)=1+\operatorname{rem}(n,m)\\
\text{if}\;\;\scal{x_n}{u_{\operatorname{i}(n)}}\leq
\eta_{\operatorname{i}(n)}\\
\left\lfloor
\begin{array}{l}
y_n=x_n\\
\end{array}
\right.\\[1mm]
\text{if}\;\;\scal{x_n}{u_{\operatorname{i}(n)}}>
\eta_{\operatorname{i}(n)}\\
\left\lfloor
\begin{array}{l}
y_n=x_n+\displaystyle{\frac{\eta_{\operatorname{i}(n)}-
\scal{x_n}{u_{\operatorname{i}(n)}}}
{\scal{u_{\operatorname{i}(n)}}{W_n^{-1}u_{\operatorname{i}(n)}}}
W_n^{-1}u_{\operatorname{i}(n)}}
\end{array}
\right.\\[1mm]
x_{n+1}=x_n+\lambda_n(y_n-x_n).
\end{array}
\right.\\[1mm]
\end{equation}
Then there exists $\overline{x}\in C$ such that
$x_n\weakly\overline{x}$.
\end{example}
\begin{proof}
Set $(\forall i\in I)$ 
$C_i=\menge{x\in\HH}{\scal{x}{u_i}\leq\eta_i}$. Then it follows 
from \cite[Example~28.16(iii)]{Livre1} that
\eqref{e:2012-08-19-md2} can be rewritten as 
\begin{equation}
\label{e:2012-08-19-md3}
(\forall n\in\NN)\quad 
x_{n+1}=x_n+\lambda_n\Big(
P_{C_{1+\operatorname{rem}(n,m)}}^{W_n}x_n-x_n\Big).
\end{equation}
The claim is therefore a consequence of 
Corollary~\ref{c:2}\ref{c:2i}.
\end{proof}

We now turn our attention to the problem of finding a zero of 
a maximally monotone operator $A\colon\HH\to 2^{\HH}$ 
(see \cite{Livre1} for background) via 
a variable metric proximal point algorithm. Let $\alpha\in\RPP$, 
let $\gamma\in\RPP$, let $W\in\BP_{\alpha}(\HH)$, and let 
$A\colon\HH\to 2^{\HH}$ be maximally monotone with graph
$\gra A$. It follows from 
\cite[Corollary~3.14(ii)]{Sico03} (applied with $f\colon
x\mapsto\scal{Wx}{x}/2$) that
\begin{equation}
\label{e:md24}
J_{\gamma A}^W\colon\HH\to\HH\colon x\mapsto(W+\gamma A)^{-1}(Wx)
\end{equation}
is well-defined, and that
\begin{equation}
\label{e:md28}
J_{\gamma A}^W\in\mathfrak{T}(W)\quad\text{and}\quad\Fix 
J_{\gamma A}^W=\menge{z\in\HH}{0\in Az}.
\end{equation}
We write $J_{\gamma A}^{\Id}=J_{\gamma A}$.  

\begin{corollary}
\label{c:3}
Let $\alpha\in\RPP$, let $A\colon\HH\to 2^{\HH}$ be a maximally 
monotone operator such that $C=\menge{z\in\HH}{0\in Az}\neq\emp$, 
let $(a_n)_{n\in\NN}$ be a sequence in $\HH$ such that 
$\sum_{n\in\NN}\|a_n\|<\pinf$, let $(\eta_n)_{n\in\NN}$ be a 
sequence in $\ell_{+}^1(\NN)$, and 
let $(W_n)_{n\in\NN}$ be a sequence in $\BP_{\alpha}(\HH)$ such 
that $\mu=\sup_{n\in\NN}\|W_n\|<\pinf$ and $(\forall n\in\NN)$ 
$(1+\eta_n)W_n\succcurlyeq W_{n+1}$.
Fix $\varepsilon\in\left]0,1\right[$ and $x_0\in\HH$, let 
$(\lambda_n)_{n\in\NN}$ be a sequence in 
$[\varepsilon,2-\varepsilon]$, let 
$(\gamma_n)_{n\in\NN}$ be a sequence in 
$\left[\varepsilon,\pinf\right[$, and set
\begin{equation}
\label{e:md08}
(\forall n\in\NN)\quad 
x_{n+1}=x_n+\lambda_n\Big(J^{W_n}_{\gamma_n A}x_n+a_n-x_n\Big).
\end{equation}
Then the following hold for some $\overline{x}\in C$.
\begin{enumerate}
\item 
\label{c:3i}
$x_n\weakly\overline{x}$.
\item 
\label{c:3ii}
Suppose that $\inte C\neq\emp$ and that there exists 
$(\nu_n)_{n\in\NN}\in\ell_{+}^1(\NN)$ such that $(\forall n\in\NN)$
$(1+\nu_n)W_{n+1} \succcurlyeq W_n$. Then $x_n\to\overline{x}$.
\item 
\label{c:3iii}
Suppose that $A$ is pointwise uniformly monotone 
on $C$, i.e., for every $x\in C$ there exists an increasing function 
$\phi\colon\RP\to\RPX$ vanishing only at $0$ such that 
\begin{equation}
\label{e:Bunifmon}
(\forall u\in Ax)(\forall (y,v)\in\gra A)
\:\;\scal{x-y}{u-v}\geq\phi(\|x-y\|).
\end{equation}
Then $x_n\to\overline{x}$. 
\end{enumerate}
\end{corollary}
\begin{proof}
In view of \eqref{e:md28}, \eqref{e:md08} is a special case of 
\eqref{e:guad2012b} with $I=\{1\}$ and $(\forall n\in\NN)$ 
$T_{1,n}=J^{W_n}_{\gamma_n A}$. 
Hence, using Theorem~\ref{t:2}\ref{t:2i}--\ref{t:2ii}, 
to show \ref{c:3i}--\ref{c:3ii}, it suffices to prove that 
\eqref{e:mdX} holds. To this end, let $(x_{p_n})_{n\in\NN}$ be 
a weakly convergent subsequence of $(x_n)_{n\in\NN}$, 
say $x_{p_n}\weakly x$, such that
$J^{W_{p_n}}_{\gamma_{p_n}A}\:x_{p_n}-x_{p_n}\to 0$. To show that
$0\in Ax$, let us set
\begin{equation}
(\forall n\in\NN)\quad 
y_n=J^{W_n}_{\gamma_n A}\:x_n
\quad\text{and}\quad 
v_n=\frac{1}{\gamma_n}W_n(x_n-y_n).
\end{equation}
Then \eqref{e:md24} yields $(\forall n\in\NN)$ $v_n\in Ay_n$. On
the other hand, since $y_{p_n}-x_{p_n}\to 0$, we have
\begin{equation}
\|v_{p_n}\|=\frac{\|W_{p_n}(x_{p_n}-y_{p_n})\|}{\gamma_{p_n}}
\leq\frac{\mu}{\varepsilon}\|x_{p_n}-y_{p_n}\|\to 0.
\end{equation}
Thus, $y_{p_n}\weakly x$ and $Ay_{p_n}\ni v_{p_n}\to 0$. Since 
$\gra A$ is sequentially closed in 
$\HH^\text{weak}\times\HH^\text{strong}$
\cite[Proposition~20.33(ii)]{Livre1}, we conclude that $0\in Ax$.
Let us now show \ref{c:3iii}. We have $0\in A\overline{x}$ and 
$(\forall n\in\NN)$ $v_{p_n}\in Ay_{p_n}$. Hence, it follows from 
\eqref{e:Bunifmon} that there exists an increasing function 
$\phi\colon\RP\to\RPX$ vanishing only at $0$ such that 
\begin{equation}
\label{e:2012-08-18-md1}
(\forall n\in\NN)\quad\scal{y_{p_n}-\overline{x}}{v_{p_n}}\geq
\phi(\|y_{p_n}-\overline{x}\|).
\end{equation}
Since $v_{p_n}\to 0$, we get
$\phi(\|y_{p_n}-\overline{x}\|)\to 0$ and, in turn, 
$\|y_{p_n}-\overline{x}\|\to 0$. It follows that
$\|x_{p_n}-\overline{x}\|\to 0$ and hence that
$\varliminf d_C(x_n)=0$. In view of Theorem~\ref{t:2}\ref{t:2iii}, 
we conclude that $x_n\to\overline{x}$.
\end{proof}

\begin{remark}
Corollary~\ref{c:3}\ref{c:3i} reduces to the classical result 
of \cite[Theorem~1]{Rock76} when $(\forall n\in\NN)$
$W_n=\Id$, $\eta_n=0$, and $\lambda_n=1$. In this context,
Corollary~\ref{c:3}\ref{c:3ii} appears in 
\cite[Section~6]{Neva79}. In a finite-dimensional setting,
an alternative variable metric proximal point algorithm 
is proposed in \cite{Pare08}, which also uses 
the above conditions on $(W_n)_{n\in\NN}$ but alternative
error terms and relaxation parameters.
\end{remark}

\section{Application to inverse problems}
\label{sec:6}

In this section, we consider an application to a structured 
variational inverse problem. Henceforth, $\Gamma_0(\HH)$ denotes
the class of proper lower semicontinuous convex functions from 
$\HH$ to $\RX$.

\begin{problem}
\label{prob:1}
Let $f\in\Gamma_0(\HH)$ and let $I$ be a nonempty finite index set.
For every $i\in I$, let $(\GG_i,\|\cdot\|_i)$ be a real Hilbert 
space, let $L_i\colon\HH\to\GG_i$ be a nonzero bounded linear 
operator, let $r_i\in\GG_i$, and let $\mu_i\in\RPP$. 
The problem is to 
\begin{equation}
\label{e:2012-08-16b}
\minimize{x\in\HH}f(x)+\frac12\sum_{i\in I}\mu_i\|L_ix-r_i\|_i^2.
\end{equation}
\end{problem}

This formulation covers many inverse problems 
(see \cite[Section~5]{Smms05} and the references therein) and 
it can be interpreted as follows:
an ideal object $\widetilde{x}\in\HH$ is to be recovered from 
noisy linear measurements $r_i=L_i\widetilde{x}+w_i\in\GG_i$, 
where $w_i$ represents noise ($i\in I$), and the function $f$
penalizes the violation of prior information on $\widetilde{x}$.
Thus, \eqref{e:2012-08-16b} attempts to strike a balance between 
the observation model, represented by the data fitting term 
$x\mapsto(1/2)\sum_{i\in I}\mu_i\|L_ix-r_i\|_i^2$, and 
a priori knowledge, represented by $f$. To solve this
problem within our framework, we require the following facts.

Let $\alpha\in\RPP$, let $W\in\BP_{\alpha}(\HH)$, and let 
$\varphi\in\Gamma_0(\HH)$. The proximity operator of $\varphi$ 
relative to the metric induced by $W$ is 
\begin{equation}
\prox^{W}_\varphi\colon\HH\to\HH\colon x
\mapsto\underset{y\in\HH}{\argmin}\bigg(\varphi(y)+
\frac12\|x-y\|_{W}^2\bigg).
\end{equation}
Now, let $\partial\varphi$ be the subdifferential of $\varphi$
\cite[Chapter~16]{Livre1}. Then, in connection with \eqref{e:md24}, 
$\partial\varphi$ is maximally monotone and we have
\cite[Section~3.3]{Varm12}
\begin{equation}
\label{e:prox3}
(\forall\gamma\in\RPP)\quad
\prox^{W}_{\gamma\varphi}=J^{W}_{\gamma\partial\varphi}=
(W+\gamma\partial\varphi)^{-1}\circ W.
\end{equation}
We write $\prox^{\Id}_{\gamma\varphi}=\prox_{\gamma\varphi}$.

\begin{lemma}
\label{l:2012-08-16}
Let $A\colon\HH\to 2^{\HH}$ be maximally monotone,
let $U$ be a nonzero operator in $\BP_0(\HH)$, let 
$\gamma\in\left]0,1/\|U\|\right[$, let $u\in\HH$, 
set $W=\Id-\gamma U$, and set $B=A+U+\{u\}$. Then 
\begin{equation}
\label{e:id2}
(\forall x\in\HH)\quad
J_{\gamma B}^{W}x=J_{\gamma A}\big(Wx-\gamma u\big).
\end{equation}
\end{lemma}
\begin{proof}
Since $U\in\BP_0(\HH)$, $U$ is maximally monotone 
\cite[Example~20.29]{Livre1}. In turn, it follows from 
\cite[Corollary~24.4(i)]{Livre1} that $B$ is maximally monotone. 
Moreover, $W\in\BP_{\alpha}(\HH)$, where $\alpha=1-\gamma\|U\|$. 
Now, let $x$ and $p$ be in $\HH$. 
Then it follows from \eqref{e:md24} that
\begin{equation}
p=J_{\gamma B}^{W}x
\Leftrightarrow Wx\in Wp+\gamma Bp
\Leftrightarrow Wx-\gamma u\in p+\gamma Ap
\Leftrightarrow p=J_{\gamma A}\big(Wx-\gamma u\big),
\end{equation}
which completes the proof.
\end{proof}

\begin{proposition}
\label{p:2012-08-16}
Let $\varepsilon\in\left]0,1/(1+\sum_{i\in I}\mu_i\|L_i\|^2)
\right[$, let $(a_n)_{n\in\NN}$ be a sequence 
in $\HH$ such that $\sum_{n\in\NN}\|a_n\|<\pinf$, let 
$(\eta_n)_{n\in\NN}$ be a sequence in $\ell_{+}^1(\NN)$, and let 
$(\gamma_n)_{n\in\NN}$ be a sequence in $\RR$ such that
\begin{equation}
\label{e:2012-08-24a}
(\forall n\in\NN)\quad 
\varepsilon\leq\gamma_n\leq
\frac{1-\varepsilon}{\displaystyle\sum_{i\in I}\mu_i\|L_i\|^2}
\quad\text{and}\quad
(1+\eta_n)\gamma_n-\gamma_{n+1}\leq\frac{\eta_n}
{\displaystyle\sum_{i\in I}\mu_i\|L_i\|^2}.
\end{equation}
Furthermore, let $C$ be the set of solutions to 
Problem~\ref{prob:1}, let $x_0\in\HH$, let $(\lambda_n)_{n\in\NN}$
be a sequence in $[\varepsilon,2-\varepsilon]$, and set
\begin{equation}
\label{e:2012-08-16md1}
(\forall n\in\NN)\quad
x_{n+1}=x_n+\lambda_n\bigg(\prox_{\gamma_n f}
\bigg(x_n+\gamma_n\sum_{i\in I}\mu_iL_{i}^*
\big(r_i-L_{i}x_n\big)\bigg)+a_n-x_n\bigg).
\end{equation}
Then the following hold for some $\overline{x}\in C$.
\begin{enumerate}
\item
\label{p:2012-08-16i}
Suppose that 
\begin{equation}
\label{e:athens2012-08-22a}
\lim_{\|x\|\to\pinf}\;\;f(x)+\frac12\sum_{i\in I}
\mu_i\|L_ix-r_i\|_i^2=\pinf.
\end{equation}
Then $x_n\weakly\overline{x}$.
\item
\label{p:2012-08-16ii}
Suppose that there exists $j\in I$ such that $L_j$ is bounded 
below, say,
\begin{equation}
\label{e:2012-08-16a}
(\exi\beta\in\RPP)(\forall x\in\HH)\quad\|L_jx\|_j\geq\beta\|x\|.
\end{equation}
Then $C=\{\overline{x}\}$ and $x_n\to\overline{x}$.
\end{enumerate}
\end{proposition}
\begin{proof}
Set $U=\sum_{i\in I}\mu_iL_i^*L_i$ and
$u=-\sum_{i\in I}\mu_iL_i^*r_i$.
Then 
\begin{equation}
\label{e:2012-08-16md8}
\|U\|\leq\sum_{i\in I}\mu_i\|L_i\|^2,
\end{equation}
and the assumptions imply that $0\neq U\in\BP_{0}(\HH)$ and that 
$(\forall n\in\NN)$ 
$\varepsilon\leq\gamma_n\leq(1-\varepsilon)/\|U\|$. Now set 
\begin{equation}
\label{e:2012-08-17md1}
g\colon\HH\to\RX\colon x\mapsto
f(x)+\frac12\scal{Ux}{x}+\scal{x}{u}
\end{equation}
and 
\begin{equation}
\label{e:2012-08-17md3}
(\forall n\in\NN)\quad W_n=\Id-\gamma_n U.
\end{equation}
Then \eqref{e:2012-08-16b} is equivalent to minimizing $g$.
Furthermore, it follows from \eqref{e:2012-08-24a} that
$(W_n)_{n\in\NN}$ lies in $\BP_{\varepsilon}(\HH)$ and
that $\sup_{n\in\NN}\|W_n\|\leq 2-\varepsilon$. In addition,
we have
\begin{equation}
\label{e:2012-08-27a}
(\forall n\in\NN)\quad\eta_n\geq
\big((1+\eta_n)\gamma_n-\gamma_{n+1}\big)\|U\|.
\end{equation}
Indeed if, for some $n\in\NN$, $(1+\eta_n)\gamma_n\leq\gamma_{n+1}$
then $\eta_n\geq 0\geq((1+\eta_n)\gamma_n-\gamma_{n+1})\|U\|$;
otherwise we deduce from \eqref{e:2012-08-24a} and
\eqref{e:2012-08-16md8} that 
$\eta_n\geq((1+\eta_n)\gamma_n-\gamma_{n+1})
\sum_{i\in I}\mu_i\|L_i\|^2\geq
((1+\eta_n)\gamma_n-\gamma_{n+1})\|U\|$. Thus, since 
$U\in\BP_0(\HH)$, we have 
$\|U\|=\sup_{\|x\|\leq 1}\scal{Ux}{x}$ and therefore
\begin{eqnarray}
\text{\eqref{e:2012-08-27a}}
&\Rightarrow&
(\forall n\in\NN)(\forall x\in\HH)\quad\eta_n\|x\|^2
\geq\big((1+\eta_n)\gamma_n-\gamma_{n+1}\big)\scal{Ux}{x}
\nonumber\\
&\Rightarrow&
(\forall n\in\NN)(\forall x\in\HH)\quad
(1+\eta_n)(\|x\|^2-\gamma_n\scal{Ux}{x})\geq\|x\|^2-
\gamma_{n+1}\scal{Ux}{x}\nonumber\\
&\Rightarrow&
(\forall n\in\NN)\quad(1+\eta_n)W_n\succcurlyeq W_{n+1}.
\end{eqnarray}
Now set $A=\partial f$ and $B=A+U+\{u\}$. Then
we derive from \cite[Corollary~16.38(iii)]{Livre1} that 
$B=\partial g$. Hence, using \eqref{e:prox3}, 
\eqref{e:2012-08-17md3}, and Lemma~\ref{l:2012-08-16}, 
\eqref{e:2012-08-16md1} can be rewritten as 
\begin{align}
\label{e:2012-08-17md2}
(\forall n\in\NN)\quad 
x_{n+1}
&=x_n+\lambda_n\Big(\prox_{\gamma_n f}
\big(x_n-\gamma_n(Ux_n+u)\big)+a_n-x_n\Big)\nonumber\\
&=x_n+\lambda_n\Big(J_{\gamma_n A}
\big(W_nx_n-\gamma_nu\big)+a_n-x_n\Big)\nonumber\\
&=x_n+\lambda_n\Big(J^{W_n}_{\gamma_n B}x_n
+a_n-x_n\Big).
\end{align}
On the other hand, it follows from 
Fermat's rule \cite[Theorem~16.2]{Livre1} that 
\begin{equation}
\menge{z\in\HH}{0\in Bz}=\operatorname{Argmin} g=C.
\end{equation}

\ref{p:2012-08-16i}: 
Since $f\in\Gamma_0(\HH)$ and $U\in\BP_{0}(\HH)$,
it follows from \cite[Proposition~11.14(i)]{Livre1} that 
Problem~\ref{prob:1} admits at least one solution. Altogether, 
the result follows from Corollary~\ref{c:3}\ref{c:3i}.

\ref{p:2012-08-16ii}:
It follows from \eqref{e:2012-08-16a} that 
$L_j^*L_j\in\BP_{\beta^2}(\HH)$. 
Therefore, $U\in\BP_{\mu_j\beta^2}(\HH)$ and, 
since $f\in\Gamma_0(\HH)$, we derive from 
\eqref{e:2012-08-17md1} that $g\in\Gamma_0(\HH)$ 
is strongly convex. Hence, \cite[Corollary~11.16]{Livre1} 
asserts that \eqref{e:2012-08-16b} possesses a unique solution,
while \cite[Example~22.3(iv)]{Livre1} asserts that $B$ is
strongly -- hence uniformly -- monotone. Altogether, 
the claim follows from Corollary~\ref{c:3}\ref{c:3iii}.
\end{proof}

\begin{remark}
\label{r:2012-08-27}
In Problem~\ref{prob:1} suppose that $I=\{1\}$, $\mu_1=1$, $L_1=L$,
and $r_1=r$, and that $\lim_{\|x\|\to\pinf}$ $f(x)+
\|Lx-r\|_1^2/2=\pinf$. Then \eqref{e:2012-08-16md1} reduces to 
the proximal Landweber method
\begin{equation}
\label{e:2012-08-27md1}
(\forall n\in\NN)\quad 
x_{n+1}=x_n+\lambda_n\Big(\prox_{\gamma_n f}
\big(x_n+\gamma_nL^*(r-Lx_n)\big)+a_n-x_n\Big),
\end{equation}
and we derive from 
Proposition~\ref{p:2012-08-16}\ref{p:2012-08-16i} that 
$(x_n)_{n\in\NN}$ converges weakly to a minimizer of
$x\mapsto f(x)+\|Lx-r\|_1^2/2$ if 
\begin{equation}
\label{e:2012-08-27md2}
(\forall n\in\NN)\quad 
\begin{cases}
\varepsilon\leq\gamma_n\leq(1-\varepsilon)/\|L\|^2\\
(1+\eta_n)\gamma_n\leq\gamma_{n+1}+\eta_n/\|L\|^2\\
\varepsilon\leq\lambda_n\leq 2-\varepsilon.
\end{cases}
\end{equation}
This result complements \cite[Theorem~5.5(i)]{Smms05}, which
establishes weak convergence under alternative conditions on the
parameters $(\gamma_n)_{n\in\NN}$ and $(\lambda_n)_{n\in\NN}$,
namely 
\begin{equation}
\label{e:2012-08-27md3}
(\forall n\in\NN)\quad 
\begin{cases}
\varepsilon\leq\gamma_n\leq(2-\varepsilon)/\|L\|^2\\
\varepsilon\leq\lambda_n\leq 1.
\end{cases}
\end{equation}
In particular, suppose that $\HH$ is separable, let
$(e_k)_{k\in\NN}$ be an orthonormal basis of $\HH$,
and set $f\colon x\mapsto\sum_{k\in\NN}\phi_k(\scal{x}{e_k})$, 
where $(\forall k\in\NN)$ $\Gamma_0(\RR)\ni\phi_k\geq\phi_k(0)=0$.
Moreover, for every $n\in\NN$, let $(\alpha_{n,k})_{k\in\NN}$ be 
a sequence in $\ell^2(\NN)$ and suppose that
$\sum_{n\in\NN}\sqrt{\sum_{k\in\NN}|\alpha_{n,k}|^2}<\pinf$.
Now set $(\forall n\in\NN)$ $a_n=\sum_{k\in\NN}\alpha_{n,k}e_k$.
Then, arguing as in \cite[Section~5.4]{Smms05}, 
\eqref{e:2012-08-27md1} becomes 
\begin{equation}
\label{e:2012-08-27md4}
(\forall n\in\NN)\quad 
x_{n+1}=x_n+\lambda_n\bigg(\sum_{k\in\NN}
\big(\alpha_{n,k}+\prox_{\gamma_n\phi_k}
\scal{x_n+\gamma_nL^*(r-Lx_n)}{e_k}\big)e_k-x_n\bigg),
\end{equation}
and we obtain convergence under the new condition 
\eqref{e:2012-08-27md2} (see also \cite{Siop07} for potential
signal and image processing applications of this result).
\end{remark}

\end{document}